\documentclass{amsart}
\usepackage[utf8]{inputenc}
\usepackage{amssymb,amsmath,amsthm,amscd}
\usepackage{mathrsfs}
\usepackage{nicefrac}
\usepackage{enumerate}

\usepackage{lmodern} 
\usepackage{nicefrac}
\usepackage[all]{xy}
\xyoption{all}

  \newcommand{\C}{\ensuremath{\mathbb{C}}}%
  \newcommand{\R}{\ensuremath{\mathbb{R}}}%
  \newcommand{\N}{\ensuremath{\mathbb{N}}}%
  \newcommand{\Z}{\ensuremath{\mathbb{Z}}}%
  \newcommand{\tore}{\ensuremath{\mathbb{T}}}%
  \newcommand{\disc}{\ensuremath{\mathbb{D}}}%
  \newcommand\scal[2]{\ensuremath{\left\langle#1,#2 \right\rangle}}%
\newcommand\besov[2]{\ensuremath{B_{#1+}^{#2}}}
\newcommand\besovect[3]{\ensuremath{B_{#1}^{#2}\left(#3\right)_+}}
\newcommand\conv{\ensuremath{\star}}

\newcommand\equival{\ensuremath{\approx}}

\newtheorem{thm}{Theorem}[section]
\newtheorem{lemma}[thm]{Lemma}

\theoremstyle{remark}
\newtheorem*{rem}{Remark}

\theoremstyle{definition}
\newtheorem{dfn}[thm]{Definition}

\begin{document}
\author{Mikael de la Salle}
\title{Operator space valued Hankel matrices}

\address{D{\'e}partement de Math{\'e}matiques et Applications \\ {\'E}cole Normale
  Sup{\'e}rieure  \\ 45 rue d'Ulm \\ 75005 Paris}

\thanks{Partially supported by ANR-06-BLAN-0015.}
\email{mikael.de.la.salle@ens.fr}
\keywords{}

\begin{abstract}
  If $E$ is an operator space, the non-commutative vector valued $L^p$
  spaces $S^p[E]$ have been defined by Pisier for any $1 \leq p \leq
  \infty$. In this paper a necessary and sufficient condition for a matrix
  of the form $(a_{i+j})_{0 \le i,j}$ with $a_k \in E$ to be bounded in
  $S^p[E]$ is established if $1\leq p<\infty$. This extends previous
  results of Peller where $E=\C$ or $E=S^p$. The main theorem states that
  if $1 \leq p < \infty$, $(a_{i+j})_{0 \le i,j}$ is bounded in $S^p[E]$ if
  and only if there is an analytic function $\varphi$ in the vector valued
  Besov Space $\besovect{p}{1/p}{E}$ such that $a_n = \widehat \varphi(n)$
  for all $n \in \N$. In particular this condition only depends on the
  Banach space structure of $E$. We also show that the norm of the
  isomorphism $\varphi \mapsto (\widehat \varphi(i+j))_{i,j}$ grows as
  $\sqrt p$ as $p \to \infty$, and compute the norm of the natural
  projection onto the space of Hankel matrices.
\end{abstract}
\maketitle

\section*{Introduction}

This paper is devoted to the study of Hankel matrices in the vector-valued
non-commutative $L^p$-space $S^p[E]$ defined by Pisier
\cite{MR1648908}. The main result is a characterisation, for any operator
space $E$, of the norm of such matrices in terms of vector-valued Besov
spaces $\besovect{p}{s}{E}$ defined in the second section. The surprising
fact is that these norms only depend on the Banach-space structure of
$E$. The main result is the following.

If $\varphi = \sum_{n\in \N} a_n z^n$ is a formal series with $a_n$
belonging to an operator space $E$, we denote $a_n = \widehat \varphi(n)$
($\widehat \varphi(n)$ coïncides with the Fourier coefficient of $\varphi$
when $\varphi \in L^1(\tore;E)$), the Hankel matrix $\Gamma_\varphi$ is
defined by its matrix representation
\[\Gamma_\varphi = \left(\widehat \varphi(j+k)\right)_{j,k \geq 0}.\]

\begin{thm}
\label{thm=Main_thm}
  Let $1 \leq p < \infty$. A Hankel matrix $(a_{j+k})_{j,k\geq 0}$ belongs
  to $S^p[E]$ if and only if the formal series $\sum_{n \geq 0} a_n z^n$
  belongs to $\besovect{p}{1/p}{E}$.

  More precisely there is a constant $C >0$ such that for any operator
  space $E$ and any formal series $\varphi = \sum_{n \geq 0} a_k z^k$
  \[C^{-1} \left\|\varphi\right\|_{\besovect{p}{1/p}{E}}\leq
  \left\|\Gamma_\varphi \right\|_{S^p[E]} \leq C \sqrt p
  \left\|\varphi\right\|_{\besovect{p}{1/p}{E}}.\]

  Moreover the rate of growth as $\sqrt p$ is optimal already in the scalar
  case: there is a constant $c>0$ (independant of $p$) and $\varphi \in
  \besov{p}{1/p}$ such that $\left\|\Gamma_\varphi \right\|_{S^p} \geq c
  \sqrt p \left\|\varphi\right\|_{\besov{p}{1/p}}$.
\end{thm}

As a consequence we also get that the norm of the natural projection onto
the space of Hankel matrices grows as $\sqrt p$ as $p\to  \infty$, and as
$1/\sqrt{p-1}$ as $p \to 1$:
\begin{thm}
\label{thm=projection}
Let $P_{Hank}$ be the natural projection from the space of infinite matrices to
the subspace of Hankel matrices:
  \[P_{Hank} \left((a_{j,k})_{j,k \geq 0}\right) = \left(\frac 1 {j+k+1} \sum_{s+t
      = j+k} a_{s,t}\right)_{j,k \geq 0}.\] 

  Then, for $1<p<\infty$, $P_{Hank}$ is bounded on $S^p$ (and on $S^p[E]$
  for any operator space $E$) and its norms satisfy the following
  inequality with a constant $C>0$ independant on $E$ and $p$:
  \[ C^{-1} \sqrt{\frac{p^2}{p-1}} \leq \|P_{Hank}\|_{S^p \to S^p} \leq
  \|P_{Hank}\|_{S^p[E] \to S^p[E]} \leq C \sqrt{\frac{p^2}{p-1}}.\]
\end{thm}

As often for results on non-commutative $L^p$ spaces Theore
\ref{thm=Main_thm} is proved using the complex interpolation method. For
$p=1$ the above theorem can be proved directly. A first natural attempt to
derive the Theorem for any $p$ would be to get something for
$p=\infty$. Bounded Hankel operators are well-known with Nehari's theorem
and its operator valued version, which states that for $E\subset B(\ell^2)$
and $p=\infty$, $\Gamma_\varphi$ belongs to $B(\ell^2)\otimes E$ if and
only if there is a function $\psi \in L^\infty(\tore;B(\ell^2))$ such that
$\widehat \psi(k) = \widehat \varphi(k)$ for $k>0$. But for non-injective
operator spaces, this seems very complicated (at least to me) to relate
this function $\psi$ to properties of $E$. Another natural attempt would be
to interpolate between $p=2$ and $p=1$ since often for $p=2$ results are
obvious. But it should be pointed out that here the Theorem is non trivial
for $p=2$ as well. We are thus led to pass from a problem with only one
parameter $p$ to a problem with more parameters to ``get room'' in order to
be able to use the interpolation method. This is done with the so-called
generalized Hankel matrices.

For real (or complex) numbers $\alpha,\beta$ the generalized Hankel
matrix with symbol $\varphi$ is defined by
\[\Gamma_\varphi^{\alpha,\beta}
= \left((1+j)^\alpha (1+k)^\beta\widehat \varphi(j+k)\right)_{j,k \geq
  0}.\]

Our main theorem characterizes, for an operator space $E$ and a $1
\leq p \leq \infty$, the generalized Hankel matrices that belong to
$S^p[E]$ under the conditions that $\alpha+1/2p>0, \beta+1/2p>0$.
\begin{thm}
\label{thm=thm_principal_Hankel_Besov}
Let $1 \leq p \leq \infty$ and $\alpha, \beta > \nicefrac{-1}{2p}$. Then
for a formal series $\varphi = \sum_{n \geq 0} \widehat \varphi(n) z^n$
with $\widehat \varphi(n) \in E$, $\Gamma_\varphi^{\alpha,\beta} \in
S^p[E]$ if and only if $\varphi \in \besovect{p}{\nicefrac 1 p + \alpha +
  \beta}{E}$.

More precisely, for all $M>0$, there is a constant $C = C_M$ (depending
only on $M$, not on $p$, E) such that for all such $\varphi$, all $1 \leq p
\leq \infty$ and all $\alpha, \beta \in \R$ such that $\nicefrac{-1}{2p} <
\alpha,\beta <M$,
  \begin{equation}
\label{eq=inegalite_principale_Hankel_Besov}
C^{-1} \left\|\varphi\right\|_{\besovect{p}{1/p + \alpha +
    \beta}{E}}\leq \left\|\Gamma_\varphi^{\alpha,\beta} \right\|_{S^p[E]}
\leq \frac{C}{\sqrt{\min(\alpha, \beta)+\nicefrac 1 {2p}}^{1+1/p}}
\left\|\varphi\right\|_{\besovect{p}{1/p + \alpha + \beta}{E}}.
\end{equation}
\end{thm}
The usual convention is to define $S^\infty[E]$ as $\mathcal K
\otimes_{min} E$. However in the previous Theorem one has to (abusively)
understand $\|\cdot\|_{S^\infty[E]}$ as $\|\cdot\|_{B(\ell^2) \otimes_{min}
E}$ (if $E$ is finite dimensional) or even as $\|\cdot\|_{B(\ell^2
\otimes H)}$ if $E \subset B(H)$.

Note that surprisingly, this theorem shows that the condition
$\Gamma_\varphi^{\alpha,\beta} \in S^p[E]$ only depends on the
Banach space structure of $E$ (whereas the Banach space structure of
$S^p[E]$ depends on the operator space structure of $E$). 

These results extend results of Peller in the scalar case or in the case
when $E=S^p$ (\cite{MR602274},\cite{MR647702},\cite{MR725454},
\cite{MR1949210}). In the scalar case Peller's theorem indeed shows that
the space of Hankel matrices in $S^p$ is isomorphic to a Besov space
$\besov{p}{1/p}$. The case when $E=S^p$ shows that this isomorphism is in
fact a complete isomorphism. The results stated above show that this
isomorphism has the stronger property of being \emph{regular} as well as
its inverse in the sense of \cite{MR1324838}. In this paper the choice was
made to use the vocabulary of regular operators, but one could easily avoid
this notion (replacing, in the proof of Lemma
\ref{thm=application_reguliere}, the use of Pisier's Theorem
\ref{thm=applications_regulieres_NC} by Stein's interpolation method). The
natural projection $P_{Hank}$ was also studied by Peller (Chapter 6 of
\cite{MR1949210}) who proved that it is bounded on $S^p$ if $1<p<\infty$
and unbounded if $p=1$ or $\infty$. Here we prove that it is even regular,
and show that its norm as well as its regular norm behaves as $\sqrt p$
($p\geq 2$) or as $1/\sqrt{p-1}$ ($p \geq 2$). This seems to be new even in
the scalar case.

These results should be considered as remarks on Peller's proof rather than
new theorems, since the steps presented here are all close to one of
Peller's proofs (\cite{MR1949210}, sections 8 and 9 of Chapter 6). There
are still some adaptations to make since for example the result for $p=2$
is non-trivial here whereas it is obvious in Peller's case. Moreover as far
as the constants in the isomorphisms are concerned, our results are more
precise and optimal in some sense (if one follows Peller's proofs, one is
led to constants growing at least as fast as $p$ in the right-hand side of
the inequality of the Theorem \ref{thm=Main_thm}). For completeness we
provide a detailed proof. We would also like to mention here the fact that
{\'E}ric Ricard has found a much shorter and elementary proof of Theorem
\ref{thm=Main_thm} (which is in particular a new simpler proof of Peller's
results), but it leads to constants of order $p$ instead of $\sqrt p$. It
is also worth mentioning that (at least one direction of) his proof also
works for $p<1$ (in the scalar and $S^p$-valued case).

Peller's classical results also have an extension to the case $0<p<1$. Here
there are some obstructions: we should first of all clarify the notion of
vector-valued non-commutative $L^p$ spaces for $p<1$. But even then, since
the proof given here really lies on duality and interpolation, some new
ideas would be needed.

This chapter is organized as follows: in the first section we recall
briefly definitions and facts on regular operators. In the second section
we give definitions and classical results on Besov spaces of analytic
functions $\besov{p,q}{s}$ that will be used later. All results are
proved. In the third and last section we prove the main result.

\subsection*{Notation} We will use the following notation: if $X$ and $Y$
are two Banach spaces (resp. operator spaces), we write $X\simeq Y$ if $X$
and $Y$ are isomorphic (resp. completely isomorphic). Most of the time the
isomorphism will not be explicited since it is natural.  If $A$ and $B$ are
two nonnegative numerical expressions (depending on some parameters), we
will write $A \equival B$ if there is a constant $c$ such that $c^{-1}
A\leq B \leq c A$.
\section{Background on regular operators}
\label{part=regular_operators}
\subsection{Commutative case}
We start by recalling the definition of regular operators in the
commutative setting.
\begin{dfn} 
A linear operator $u:\Lambda_1 \to \Lambda_2$ between Banach lattices
is said to be regular if for any Banach space $X$, $u \otimes id_X :
\Lambda_1(X) \to \Lambda_1(X)$ is bounded. Equivalently (taking for
$X=\ell^\infty_n$), if there is a constant $C$ such that for any $n$
and $f_1,\dots, f_n \in \Lambda_1$,
\[\left\|\sup_k |u(f_k)|\right\|_{\Lambda_2} \leq C
  \left\|\sup_k |f_k|\right\|_{\Lambda_1}.\]

The smallest such $C$ is denoted by $\|u\|_r$.
\end{dfn}
This theory applies in particular if $\Lambda_1$ and $\Lambda_1$ are
(commutative) $L^p$ spaces: when $p=1$ or $p = \infty$ a map is regular if
and only if it is bounded. Similarly, a map that is simultaneously bounded
$L^1 \to L^1$ and $L^\infty \to L^\infty$ is regular on $L^p$. This is not
far from being a characterization since it is known that the set of regular
operators: $L^p \to L^p$ coincides with the interpolation space (for the
second complex interpolation method) between $B(L^\infty,L^\infty)$ and
$B(L^1,L^1)$.

We refer to \cite{MR0482275} for facts on the complex interpolation
method.

\subsection{Non-commutative case}
Let $S$ be a subspace of a non-commutative $L^p$ space constructed on a
hyperfinite von Neumann algebra. In the sequel for an operator space $E$ we
will denote by $S[E]$ the (closure of) the subspace $S \otimes E$ of the
vector valued non-commutative $L^p$-space $L^p(\tau;E)$ defined in
\cite{MR1648908}.
\begin{dfn} 
A linear map $u:S \to T$ between subspaces of non-commutative $L^p$
spaces as above is said to be regular if for any operator space $E$,
$u \otimes id_E : S[E] \to T[E]$ is bounded. As in the commutative
case $\|u\|_r$ will denote the best constant $C$ such that $\|u
\otimes id_E\|_{S[E] \to T[E]}\leq C$ for all $E$.

The set of regular operators equipped with this norm will be denoted
by $B_r(S,T)$.
\end{dfn}
 Since classical $L^p$ spaces are special cases of non-commutative
 $L^p$ spaces, this notion applies also for commutative $L^p$ spaces
 (but fortunately the two notions coincide).  This notion was defined
 and studied in \cite{MR1324838}. In particular the following result
 was proved:
\begin{thm}[Pisier] 
\label{thm=applications_regulieres_NC}
Let $(\mathcal M,\tau)$ and $(\mathcal N, \widetilde \tau)$ be hyperfinite
von Neumann algebras with normal semi-finite faithful traces. Then a map
$u:L^p(\tau) \to L^p (\widetilde \tau)$ is regular is and only if it is a
linear combination of bounded completely positive operators. Moreover
isomorphically (with constant not depending on $p$ or on $\mathcal
M,\mathcal N$)
\[ B_r(L^p,L^p) \simeq \left[CB(L^\infty,L^\infty),
  CB(L^1,L^1)\right]^\theta \textrm{ for } \theta = 1/p.
\]
\end{thm}
We will only apply this fact in the case of von Neumann algebras that are
either commutative or equal $B(\ell^2)$ equipped with the usual trace. The
following result was also proved:
\begin{thm}
\label{thm=regularity_au_dual}
Let $1\leq p<\infty$. Then $u:L^p(\tau) \to L^p(\widetilde \tau)$ is
regular if and only if $u^*:L^{p'}(\widetilde \tau) \to L^{p'}(\tau)$ is
regular, and $\|u\|_r = \|u^*\|_r$.
\end{thm}
\section{Vector valued Besov spaces}
\label{part=Espaces_de_Besov}
In this section we introduce the Besov spaces of analytic functions
$\besov{p,q}{s}$. Before that we need some facts on Fourier
multipliers. Everything in this section is classical (the results are
stated in \cite{MR1949210}, and they are proved for the real line
instead of the unit circle in \cite{MR0482275}), but we give precise
proofs in order to get quantitative bounds on the norms of the
different isomorphisms.

\subsection{Fourier Multipliers on the circle}
Here $\tore$ will denote the unit circle: $\tore=\{z \in \C, |z|=1\}$
and will be equipped with its Haar probability measure.

The Fourier multiplier with symbol $(\lambda_k)_{k \in \Z}$
($\lambda_k \in \C$) is the linear map on the polynomials in $z$ and
$\overline z$ denoted by $M_{(\lambda_k)_k}$ and mapping $\sum_{k \in
  \Z} a_k z^k$ to $\sum_{k \in \Z} \lambda_k a_k z^k$. For $1 \leq p
\leq \infty$ we say that the Fourier multiplier is bounded on $L^p$ if
the map $M_{(\lambda_k)_k}$ can be extended to a bounded operator on
$L^p(\tore)$ such that for $f \in L^p(\tore)$,
$g=M_{(\lambda_k)_k}(f)$ satisfies $\widehat g(k) = \lambda_k \widehat
f(k)$.

Similarly if $X$ is a Banach space the multiplier $M_{(\lambda_k)_k}$
is said to be bounded on $L^p(\tore;X)$ if $M_{(\lambda_k)_k} \otimes
id_X$ extends to a continuous map on $L^p(\tore;X)$ (which we still
denote by $M_{(\lambda)_k)}$), such that for $f \in L^p(\tore;X)$,
 $g=(M_{(\lambda_k)_k}\otimes id_X)(f)$ satisfies $\widehat g(k) =
\lambda_k \widehat f(k)$.

In the vocabulary of part \ref{part=regular_operators} a multiplier
$M_{(\lambda_k)_k}$ is said to be regular on $L^p$ if it is bounded on
$L^p(\tore;X)$ for any Banach space $X$.

For example if $\lambda_k = \widehat \mu(k)$ for some complex Borel
measure $\mu$ on $\tore$ then $M_{(\lambda_k)_k}$ is bounded on
$L^p(\tore;X)$ ($1\leq p \leq \infty$) for any Banach space $X$ since
it corresponds to the convolution map $f \mapsto \mu \conv f$. Its
regular norm on $L^p$ is therefore equal to the total variation of
$\mu$.

The following Lemma will be essential. 
\begin{lemma}
\label{thm=multiplicateurs_fourier}
Let $\lambda = (\lambda_k)_{k \in \Z} \in \C^\Z$ satisfying $\|
\lambda \|_2<\infty$. Then the Fourier multiplier with symbol
$\lambda$ is bounded on every $L^p$ and
\[ \left\|M_{(\lambda_k)_k}\right\|_{L^p \to L^p} \leq \frac{2}{\sqrt
  \pi} \sqrt{\|\lambda \|_2 \| (\lambda_{k+1} - \lambda_k)_k\|_2}.\]
It is even regular and its regular norm on $L^p$ is less than 
\[2/\sqrt
\pi \sqrt{\|\lambda \|_2 \| (\lambda_{k+1} - \lambda_k)_k\|_2}.\]
\end{lemma}
\begin{proof}
  Since $\| (\lambda_k) \|_2<\infty$, the function $f:z \mapsto
  \sum_{k \in \Z} \lambda_k z^k$ is in $L^2$ and $\|f\|_2 = \|
  (\lambda_k) \|_2$. Similarly, the function $g:z \mapsto (1-z)f(z)$
  satisfies $\|g\|_2 = \|(\lambda_k - \lambda_{k+1})_{k \in \Z}\|_2$.

  Since the multiplier with symbol $(\lambda_k)$ corresponds to the
  convolution by $f$, by the remark preceding the Lemma we only have to
  prove that $\|f\|_1^2 \lesssim \| f\|_2 \|g\|_2$. But for any $0<s<1/2$:
\begin{eqnarray*}
  \|f\|_1 & = & \int_{0}^{1} |f(e^{2i\pi t})| d t\\
  & =&\int_{-s}^s |f(e^{2i\pi t})| d t + \int_{s}^{1-s} \frac{1}{|1-e^{2i\pi
      t}|} |(1-e^{2i\pi
    t})f(e^{2i\pi t})| d t\\
  & \leq & \sqrt{2s} \|f\|_2 + \sqrt{\int_s^{1-s} \frac{1}{|1-e^{2i\pi
        t}|^2} d t} \|g\|_2
\end{eqnarray*}
by the Cauchy-Schwarz inequality. The remaining integral can be computed:
\begin{eqnarray*}
  \int_s^{1-s} \frac{1}{|1-e^{2i\pi
      t}|^2} d t & = &2 \int_s^{1/2} \frac{1}{4\sin^2(\pi t)} dt\\
  & = & \frac{1}{2}\left[\frac{-\cos(\pi t)}{\pi \sin(\pi
      t)}\right]_s^{1/2} = \frac{1}{2 \pi \tan(\pi s)} \leq
  \frac{1}{2\pi^2 s}
\end{eqnarray*}
where we used that $\tan x \geq x$ for all $0\leq x \leq \pi/2$.
Taking $s= \|g\|_2/ 2\pi\|f\|_2 \leq 1/2$ we get the desired inequality.
\end{proof}
The following consequence will be also used a lot:
\begin{lemma}
\label{thm=multiplicateurs_de_Fourier_restreints}
  Let $I=[a,b] \subset \Z$ be an interval of size $N$ and take
  $(\lambda_k)_{k \in \Z}\in \C^\Z$.

  Then for any $1\leq p \leq \infty$, any Banach space $X$ and any $f \in
  L^p(\tore;X)$ such that $\widehat f$ is supported in $I$,
  \begin{equation}
\label{eq=restriction_de_mult_borne}
 \left\|M_{(\lambda_k)_k} f \right \|_{L^p(\tore;X)} \leq 2 \|f\|_p
 \max \left(\sup_{k \in I} |\lambda_k|, \sqrt{N\sup_{k \in I}
   |\lambda_k | \sup_{a \leq k < b} |\lambda_k -
   \lambda_{k+1}|}\right).
\end{equation}

In other words, the restriction of the multiplier $M_\lambda$ to the
subspace of $L^p(\tore)$ of functions with Fourier transform vanishing
outside of $I$ has a regular norm less than the right-hand side of
this inequality.
\end{lemma}
\begin{proof}
  Consider the multiplier $M_{\mu}$ with symbol $(\mu_k)_{k \in \Z}$
  where $\mu_k = \lambda_k$ if $k \in I$, $\mu_k=0$ if $k\leq a-N$ or
  if $k \geq b+N$, and $\mu_k$ is affine on the intervals $[a-N,a]$
  and $[b,b+N]$.

  Since $M_\mu$ and $M_\lambda$ coincide on the space of functions such
  that $\widehat f(k)=0$ for $k \notin I$, the claim will follow from the
  fact that the regular norm of $M_{\mu}$ is less that the right-hand side
  of \eqref{eq=restriction_de_mult_borne}. For this we use Lemma
  \ref{thm=multiplicateurs_fourier}, so we have to dominate $\|(\mu_k)\|_2$
  and $\|(\mu_{k+1} -\mu_k)\|_2$. Since both sequences $(\mu_k)_k$ and
  $(\mu_{k+1} - \mu_k)_k$ are supported in $]a-N,b+N]$ which is of size
  less than $3N$, their $\ell^2$-norm is less than $\sqrt{3N}$ times their
  $\ell^\infty$ norm. The inequality $\sup_k |\mu_k| \leq \sup_{k \in I}
  |\lambda_k|$ is obvious by definition of $\mu_k$. On the other hand we
  have $|\mu_{k+1}-\mu_k| = |\lambda_{k+1}-\lambda_k|$ if $k \in [a,b[$,
  and $|\mu_{k+1}-\mu_k| \leq \sup_{k \in I} |\lambda_k|/N$ otherwise since
  $\mu_k$ is affine on the intervals of size $N+1$ $[a-N,a]$ and $[b,b+N]$.

Thus by Lemma \ref{thm=multiplicateurs_fourier},
\[ \left\|M_{\mu}\right\|_{L^p(\tore;X) \to L^p(\tore;X)}
\leq \frac{2\sqrt 3}{\sqrt \pi} \max\left(\sup_{k \in I} |\lambda_k|,\sqrt{N\sup_{k \in [a,b[} |\lambda_k) | \sup_{k \in      I} |\lambda_k - \lambda_{k+1}|}\right).\]
This concludes the proof since $3 \leq \pi$.
\end{proof}

For all $n \in \N$, $n>0$ we define the function $W_n$ on $\tore$ by
\[\widehat{W_n}(k)= \left\{ 
\begin{array}{ll}
2^{-n+1}(k-2^{n-1}) & \textrm{if }2^{n-1} \leq k \leq 2^n\\
2^{-n}(2^{n+1} - k) & \textrm{if }2^{n} \leq k \leq 2^{n+1}\\
0 & \textrm{otherwise.} 
\end{array}\right.\]
We also define $W_0(z) =z + 1$.

Note that for all $k \in \N$, $\sum_{n \in \N} \widehat{W_n}(k) = 1$
(finite sum).

Since for $n >0$, $\|(\widehat W_n(k))_k \|_2 \leq \sqrt{2^n}$ and $\|
(\widehat W_n(k) - \widehat W_n(k+1))_k\|_2=\sqrt{3/2^{n}}$, Lemma
\ref{thm=multiplicateurs_fourier} implies the multiplier $f \mapsto
W_n \conv f$ has regular norm less than $2 \sqrt{3/\pi} \leq 2$ on
$L^p(\tore)$ any $1\leq p \leq \infty$.  The same is obvious for
$W_0$.

\subsection{Besov spaces of vector-valued analytic functions}
We define the $X$-valued weighted $\ell_p$ spaces $\ell_p^s(\N;X)$ for
$p>0$, $s \in \R$ and a Banach space $X$ as the space of sequences
$(x_n)_{n \in \N} \in X^\N$ such that $\| (x_n)_n\|_{\ell_p^s(\N;X)} =
\left \| (2^{ns} \|x_n\|_X)_{n \in \N} \right \|_p <\infty$.

We will deal in this paper with Besov spaces of ``analytic
functions'', which are defined in the following way. First note that
the reader should take the term ``analytic'' with care. Elements of
the Besov spaces are indeed defined as formal series $\sum_{k \geq 0}
x_k z^k$ with $z \in \tore$. The term analytic means that the formal
series are indexed by $\N$ and not $\Z$ (in particular this has
nothing to do with analytic maps defined on the real analytic manifold
$\tore$).

Let $X$ be a Banach space; $p,q >0$ and $s$ real numbers. The Besov space
$\besovect{p,q}{s}{X}$ is defined as the space of formal series
$f(z)=\sum_{k \in \N} x_k z^k$ with $x_k \in X$ such that $(2^{ns}\|W_n
\conv f\|_p)_{n\in \N} \in \ell_q$, with the norm $\left \| (2^{ns}\|W_n
  \conv f\|_p)_{n\in \N} \right\|_q$. Here by $W_n \conv f$ we mean the
(finite sum) $\sum_{k\geq 0} \widehat W_n(k) x_k z^k$, and this coincides
with the obvious notion when $f \in L^1(\tore;X)$. When $X=\C$ the Besov
space $\besovect{p,q}{s}{X}$ is simply denoted by $\besov{p,q}{s}$.

\begin{rem}[Elements of $\besovect{p,q}{s}{X}$ as functions] It is easy to
  see that when $s>0$, any $f \in \besovect{p,q}{s}{X}$ corresponds to a
  function belonging to $L^p(\tore;X)$ (and therefore also to
  $L^1(\tore;X)$). In this case the series $\sum_{n \geq 0} W_n \conv f$
  indeed converges in $L^p(\tore;X)$ (because $\sum_{n \geq 0} \|W_n \conv
  f\|_p<\infty$). It is also immediate to see that for any $s$, $\|x_k\|_X
  \leq C \|f\|_{\besovect{p,q}{s}{X}} k^{-s}$ for some constant $C>0$, and
  thus that for any $f \in \besovect{p,q}{s}{X}$, $\sum_{k \geq 0} x_k z^k$
  converges for all $z$ in the unit ball $\disc$ of $\C$.

  On the opposite when $s<0$ there are elements $f =\sum_{k \geq 0}x_k z^k
  \in \besovect{p,q}{s}{X}$ such that the sequence $x_k$ is not even
  bounded (and thus cannot represent a function in $L^1(\tore;X)$).
\end{rem}

The space can be equivalently defined as a subspace of
$\ell_q^s(\N;L^p(\tore;X))$ with the isometric injection
\begin{eqnarray*}
\besovect{p,q}{s}{X} & \longrightarrow &
 \ell_q^s(\N;L^p(\tore;X))\\ f & \mapsto & (W_n \conv f)_{n \in \N}
\end{eqnarray*}
Moreover the image of $\besovect{p,q}{s}{X}$ in the isometric
injection is a complemented subspace. The complementation map is given
by
\begin{eqnarray*}
  P: \ell_q^s(\N;L^p(\tore;X)) & \longrightarrow &
  \besovect{p,q}{s}{X}\\ (a_n) & \mapsto & (W_0+W_1) \conv a_0 +
  \sum_{n \geq 1} \left(W_{n-1} + W_n + W_{n+1}\right) \conv a_n
\end{eqnarray*}
and has norm less than $C 2^{2|s|}$ for some constant $C\leq 20$. Indeed,
if $V_n = W_{n-1} + W_n +W_{n+1}$ if $n \geq 1$ and $V_0 = W_0 + W_1$, then
$W_m \conv V_n = 0$ if $|n - m| >2$, and moreover if $|n-m|\leq 2$, $\|(W_m
\conv V_{n}) \conv a_{n}\|_p \leq 4 \|a_n\|_p$ by Lemma
\ref{thm=multiplicateurs_fourier}. This implies that
\begin{eqnarray*} \left\|\sum_{n \geq 0} V_n \conv a_n
  \right\|_{\besovect{p,q}{s}{X}} &\leq& \sum_{-2 \leq \epsilon\leq 2} 4
  \left \| (2^{ns}\|a_{n+\epsilon}\|_p)_{n\in \N} \right\|_q\\
  &\leq & 4\left(2^{-2s} +2^{-s} +1 +2^{s}+2^{2s}\right)\left \|
    (2^{ns}\|a_{n+\epsilon}\|_p)_{n\in \N} \right\|_q.
\end{eqnarray*}

When $p=q$, the Besov space $\besovect{p,q}{s}{X}$ is also denoted by
$\besovect{p}{s}{X}$. In this case $\besov{p}{s}$ is a subspace of
$\ell_p^s(\N;L^p(\tore))$ which is just the $L^p$ space of $\N \times
\tore$ with respect to the product measure of the Lebesgue measure on
$\tore$ and the measure on $\N$ giving mass $2^{nsp}$ to
$\{n\}$. Moreover (at least for $p<\infty$) $\besovect{p}{s}{X}$ is
the closure of $\besov{p}{s} \otimes X$ in the vector-valued $L^p$
space $L^p(\N \times \tore;X)$. This will allow to speak of regular
operators between $\besov{p}{s}$ and an other (subspace of a)
non-commutative $L^p$ space. Note in particular that the above remark
shows that $\besov{p}{s}$ is a complemented subspace of $L^p(\N \times
\tore)$ and that the complementation map $P$ (which does not depend on
$p$) is regular.

As a consequence of the complementation, we have the following
property of Besov spaces:
\begin{thm}
\label{thm=proprietes_des_besov}
The properties of the Besov spaces with respect to duality are: if
$p,q<\infty$
\[ \besovect{p,q}{s}{X}^* \simeq \besovect{p',q'}{-s}{X^*}\] isomorphically
for the natural duality $\langle f,g\rangle = \sum_{n \geq 0} \langle
\widehat f(n),\widehat g(n) \rangle$. Moreover for $M>0$ and any $|s|<M$
the constants in this isomorphism depend only on $M$.
\end{thm}
\begin{proof} 
  The boundedness of $P$ formally implies that the dual of
  $\besovect{p,q}{s}{X}$ is isomorphically identified with the set of
  formal series $g(z)=\sum_k \widehat g(k) z^k$ ($\widehat g(k) \in X^*$)
  equipped with the norm coming from the embedding $P^*:g \mapsto (V_n
  \otimes g)_n \in \ell_{q'}^{-s}(\N;L^{p'}(\tore;X^*))$. But the same
  argument as in the proof of the boundedness of $P$ shows that (up to
  constants depending only on $M$ if $|s|<M$)
\[\left\|(V_n \otimes g)_n\right\|_{\ell_{q'}^{-s}(\N;L^{p'}(\tore;X^*))} 
\equival \left\|(W_n \otimes
  g)_n\right|\|_{\ell_{q'}^{-s}(\N;L^{p'}(\tore;X^*))} =
\|g\|_{\besovect{q'}{-s}{X^*}}.\]
\end{proof}
For a real (or complex) number $\alpha$ and an integer $n$, we define the
number $D_n^\alpha$ by $D_0^\alpha = 1$ and for $n \geq 1$,
\[D_n^\alpha = \frac{(\alpha + 1) (\alpha + 2) \dots (\alpha + n)}{n!} =
\prod_{j=1}^n \left(1+\frac \alpha j \right).\]

For any $t \in \R$, we define the maps $I_t$ and $\widetilde I_t$ by
\[ I_t (\sum_{k \geq 0} a_k z^k) = \sum_{k \geq 0} (1+k)^t a_k z^k.\]
\[ \widetilde I_t (\sum_{k \geq 0} a_k z^k) = \sum_{k\geq 0} D_k^t a_k z^k.\]

The boundedness properties of the maps $I_t$ and $\widetilde I_t$ are
described by the following result:
\begin{thm}
\label{thm=dilatation_borne_sur_les_Besov}
 Let $M>0$ be a real number. There is a constant $C = C_M$
  (depending only on $M$) such that for any $1 \leq p,q \leq \infty$, any
  $|t| \leq M$, any $s \in \R$, and any Banach space $X$,
  \[\|I_t: \besovect{p,q}{s}{X} \to \besovect{p,q}{s-t}{X}\|,
  \|I_t^{-1}: \besovect{p,q}{s-t}{X} \to \besovect{p,q}{s}{X}\| \leq C.\]

Moreover if $-1/2 \leq t \leq M$,
\[\|\widetilde I_t: \besovect{p,q}{s}{X} \to \besovect{p,q}{s-t}{X}\|,
\|\widetilde I_t^{-1}: \besovect{p,q}{s-t}{X} \to \besovect{p,q}{s}{X}\|
\leq C.\]
\end{thm}
\begin{proof}
  Fix $M>0$ (and even $M \geq 1$) and take $|t|\leq M$. Let us treat
  the case of $I_t$. Let $f = \sum_{k \geq 0} a_k z^k \in
  \besovect{p,q}{s}{X}$. Since the maps $f \mapsto W_n \conv f$ and $f
  \mapsto I_t f$ are both multipliers, they commute, and we have that
\[\left\|I_t f\right\|_{\besovect{p,q}{s-t}{X}} = \left \|
  (2^{|n|s}\|I_t/2^{nt} (W_n \conv f)\|_p)_{n\in \N} \right\|_q.\] To
show that $\|I_t\|\leq C$, it is therefore enough to show that the
multiplier $I_t/2^{nt}$ (the symbol of which is $((1+k)/2^n)^t$) is
bounded by some constant $C$ on the subspace of $L^p(\tore,X)$
consisting of functions whose Fourier transform is supported in
$]2^{n-1},2^{n+1}[$. This follows from Lemma
    \ref{thm=multiplicateurs_de_Fourier_restreints}. We indeed have
    $((1+k)/2^n)^t \leq 2^{|t|}$ for $k \in ]2^{n-1},2^{n+1}[$. To
    dominate the difference $|((2+k)/2^n)^t - ((1+k)/2^n)^t|$ for
    $2^{n-1} < k < 2^{n+1}-1$, just dominate the derivative of $x
    \mapsto (x/2^n)^t$ on the interval $[2^{n-1},2^{n+1}]$ by $|t|
    2^{|t-1|}/2^n \leq M 2^{M+1}/2^n$.  The multiplier $I_t/2^{nt}$ is
    thus bounded by $4 \sqrt M 2^M$.

This shows that 
  \[\|I_t: \besovect{p,q}{s}{X} \to \besovect{p,q}{s-t}{X}\| \leq 4 \sqrt M 2^M
  \]
Since $I_{-t} = {I_t}^{-1}$, the inequality for $I_{-t}$ follows.

By the same argument, to dominate the norms of $\widetilde I_t$ (resp.
its inverse), we have to get a uniform bound on $\sup_k |\lambda_k|$
and $2^n \sup_k |\lambda_{k+1} - \lambda_k|$ where $\lambda_k =
D_k^t/2^{nt}$ (resp. $\lambda_k = 2^{nt}/D_k^t$). This amounts to
showing that there is a constant $C(M)$ (depending on $M$ only) such
that $1/C(M) \leq |D_k^t/2^{nt}| \leq C(M)$ and $|D_{k+1}^t/2^{nt} -
D_k^t/2^{nt}| \leq C(M)/2^n$ for $2^{n-1}\leq k < 2^{n+1}$ (the
inequality $|2^{nt}/D_{k+1}^t - 2^{nt}/D_k^t| \leq C(M)^3/2^n$ will
follow from the formula $|1/x - 1/y| = |y-x|/|xy|$). The first
inequality can be proved by taking the logarithm, noting that
$\log(1+t/j) = t/j + O(1/j^2)$ up to constants depending only on $M$
if $-1/2\leq t \leq M$, and remembering that $\sum_1^N 1/j = \log N +
O(1)$. The second inequality follows easily since $D_{k+1}^t - D_k^t
=t/(k+1) D_k^t$.
\end{proof}

We also use the following characterization of Besov spaces of analytic
vector-valued functions:
\begin{thm} 
\label{thm=besov_comme_analytiques}
Let $M >0$. Then there is a constant $C = C_M$ (depending only
  on $M$) such that for all $0<s<M$, for all Banach spaces $X$, all $1
  \leq p \leq \infty$ and all $f:\tore \to X$,
  \[ C^{-1} \|f\|_{\besovect{p,p}{-s}{X}} \leq \left\|(1-|z|)^{s-1/p}
    f\right\|_{L^p(\disc,dz;X)} \leq \frac C s
  \|f\|_{\besovect{p,p}{-s}{X}}.\]
\end{thm}
\begin{proof} The left-hand side inequality is easier. For any $0<r<1$, let
  $f_r$ denote the function $f_r(\theta)=f(re^{i\theta})$. Then
\[\left\|(1-|z|)^{s-1/p}
  f\right\|_{L^p(\disc,dz;X)} = \left(\int_0^1 (1-r)^{ps - 1}
  \|f_r\|_p^p r dr \right)^{1/p}.\] Let $1-2^{-n} \leq r \leq
  1-2^{-n-1}$ with $n \geq 1$. Then $\|f_r\|_p \geq \|W_n \conv
  f_r\|_p/2$. But $f$ is the image of $f_r$ by the multiplier with
  symbol $(r^{-k})_{k \in \Z}$. Note that for $2^{n-1}\leq k \leq
  2^{n+1}$, $r^{-k} \leq 2^4$, and for $2^{n-1}\leq k < 2^{n+1}$,
  $r^{-k-1}-r^{-k} = (1-r)r^{-k-1} \leq 2^{-n+1}2^4=2^{-n+5}$. Thus
  since multipliers commute and since the Fourier transform of $W_n
  \conv f$ vanishes outside of $]2^{n-1},2^{n+1}[$, Lemma
    \ref{thm=multiplicateurs_de_Fourier_restreints} implies
    \[ \|W_n \conv f \|_p \leq 2 \|W_n \conv f_r\|_p 2^5 \leq 2^6
    \|f_r\|_p.\] Moreover $(1-r)^{ps -1} \geq 2^{-ps}
    2^{-nsp+n}$. Integrating over $r$, we thus get that for $n\geq 1$:
\[ 2^{-nsp} \|W_n \conv f \|_p^p \leq C^{p}
\int_{1-2^{-n}}^{1-2^{-n-1}} (1-r)^{ps - 1} \|f_r\|_p^p r dr\] 
where $C$ depends only on $M$.
For $n=0$ the same inequality is very easy. Summing over $p$ and
taking the $p$-th root, we get the first inequality
\[ \|f\|_{\besovect{p,p}{-s}{X}} \leq C \left\|(1-|z|)^{s-1/p}
    f\right\|_{L^p(\disc,dz;X)}.\]

For the right-hand side inequality, note that since $\sum_n \widehat
W_n(k)=1$ for all $k \geq 0$, we have that for any $r>0$
\[ \|f_r\|_p \leq \sum_{n \geq 0} \|W_n \conv f_r\|_p.\]
Then as above since $W_n \conv f_r$ is the image of $W_n \conv f$ by
the Fourier multiplier of symbol $r^k$, Lemma
\ref{thm=multiplicateurs_de_Fourier_restreints} again implies than
\[\|W_n \conv f_r\|_p \leq 2 r^{2^{n-1}} \max(1 ,\sqrt{2^{n+1} (1-r)})
\|W_n \conv f\|_p.\]
If $m$ is such that $1-2^{-m}\leq r \leq 1-2^{-m-1}$ then
\[r^{2^{n-1}} = \left((1-2^{-m-1})^{2^{m+1}}\right)^{2^{n-m-2}} \leq
e^{-2^{n-m-2}}\] and
\[\max(1 ,\sqrt{2^{n+1} (1-r)}) \leq \max(1,\sqrt 2^{n+1-m}).\]
If for $k \in \Z$ one denotes $b_k = 2 e^{-2^{k-2}} \max(1,\sqrt 2^{k+1})
2^{ks}$ one thus has
\[\|W_n \conv f_r\|_p \leq 2^{ms} b_{n-m} 2^{-ns} \|W_n \conv f_r\|_p.\] If
$a_n = 2^{-ns} \|W_n \conv f_r\|_p$ for $n\geq 0$ and $a_n=0$ if $n<0$,
summing the previous inequality over $n$ we thus get
\[ \|f_r\|_p \leq 2^{ms} \sum_{n \geq 0} b_{n-m} a_n = 2^{ms} (a \conv
b)_m.\] Let us raise this inequality to the power $p$, multiply by
$r(1-r)^{ps -1} \leq 2^{-mps} 2^{m+1}$ and integrate on
$[1-2^{-m},1-2^{-m-1}]$. One gets
\[\int_{1-2^{-m}}^{1-2^{-m-1}} (1-r)^{ps - 1}\|f_r\|_p^p r dr \leq (a \conv
b)_m^p.\]
Summing over $m$ this leads to 
\[ \left\|(1-|z|)^{s-1/p} f\right\|_{L^p(\disc,dz;X)} \leq \left( \sum_{m
    \geq 0}(a \conv b)_m^p \right)^{1/p} \leq \|a \conv b
\|_{\ell^p(\Z)}.\] Now note that $\|a \conv b \|_{\ell^p(\Z)} \leq \|a\|_p
\|b\|_1 = \|f\|_{\besovect{p,p}{-s}{X}} \|b\|_1$. We are just left to prove
that $b \in \ell^1(\Z)$ and $\|b\|_1 \leq C/s$ with some constant $C$
depending only on $M$. If $k \geq 0$, we have $|b_k| \leq 2\sqrt 2
e^{-2^{k-2}} 2^{k (M+1/2)}$ which proves that $\sum_{k \geq 0} b_k\leq C_1$
for some constant depending only on $M$. If $k<0$, $|b_k| \leq 2^{ks+1}$,
which proves that $\sum_{k<0} |b_k| \leq 2/(2^s-1) \leq C_2/s$ for some
universal constant. This concludes the proof.
\end{proof}

When $p=2$ and $X$ is a Hilbert space, the preceding result can be
made more precise and more accurate (as $s \to 0$). This will be used
later and was mentionned to the author by Quanhua Xu:
\begin{thm}
\label{thm=besov_comme_analytiques_casHilbert}
Let $M >0$ and $X$ be a Hilbert space. Then for $-M \leq s \leq M$ and for
all $f =\sum_k a_k z^k \in \besovect{2,2}{-s}{X}$,
  \[ \|f\|_{\besovect{2,2}{-s}{X}} \equival 
    \left(\sum_{k=0}^\infty \|a_k\|^2 (1+k)^{-2s}\right)^{1/2} \equival
    \sqrt s \left\|(1-|z|)^{s-1/2} f\right\|_{L^2(\disc,dz;X)}\]
up to constants depending only on $M$.
\end{thm}
\begin{proof}
The first inequality is obvious: indeed, since $X$ is a Hilbert space,
for any integer $n$ we have
\[ \|W_n \conv f\|_{L^2(\tore;X)}^2 = \sum_{k} \widehat{W_n}(k)^2 \|a_k\|^2.\]

For the second inequality everything can be computed explicitely:
\begin{eqnarray*}\left\|(1-|z|)^{s - 1/2 }
  f\right\|_{L^2(\disc,dz;H)}^2 &=& \int_0^1 (1-r)^{2s-1} \sum_{k \geq
  0} \|a_k\|^2 r^{2k+1} dr\\
& = & \sum_{k \geq 0} \|a_k\|^2 \int_0^1 (1-r)^{2s-1}
r^{2k+1} dr.
\end{eqnarray*}
Integrating by parts $2k+1$ times, one gets
\[\int_0^1 (1-r)^{2s-1}
r^{2k+1} dr = \frac{(2k+1)2k (2k-1)\dots 1}{2s (2s+1) \dots (2 s + 2k
  + 1)} = \frac{1}{2s D_{2k+1}^{2s}}.\] Note that $D_{2k+1}^{2s}
\equival (1+k)^{2s}$ uniformly in $k$ and $s$ as long as $|s|<M$. This implies 
\[\left\|(1-|z|)^{s - 1/2 } f \right\|_{L^2(\disc,dz;H)}^2 \equival
\frac 1 s \sum_k (1+k)^{-2s} \|a_k\|^2,\] which concludes the proof.
\end{proof}

The following also holds:
\begin{thm} 
\label{thm=besov_comme_analytiques_avec_derivees}
Let $M >0$. Then there is a constant $C = C_M$ (depending only
  on $M$) such that for all $-1<s<M$, for all Banach spaces $X$, all $1
  \leq p \leq \infty$ and all $f:\tore \to X$,
  \[ C^{-1} \|f\|_{\besovect{p,p}{-s}{X}} \leq |f(0)| + \left
    \|(1-|z|)^{1+s-1/p} f'\right\|_{L^p(\disc,dz;X)} \leq \frac C {1+s}
  \|f\|_{\besovect{p,p}{-s}{X}}.\]
\end{thm}
\begin{proof}
By Theorem \ref{thm=besov_comme_analytiques}, it is enough to show that
\[ \|f\|_{\besovect{p,p}{-s}{X}} \equival |f(0)| +
\|f'\|_{\besovect{p,p}{-s-1}{X}}\]
up to constants depending only on $M$ if $|s|<M$.

Since $\|f\|_{\besovect{p,p}{-s}{X}} \equival |f(0)| +
\|f-f(0)\|_{\besovect{p,p}{-s}{X}}$, one can assume that $f(0)=0$.

But since $I_1 g=(zg)'$ for any $g$, Theorem
\ref{thm=dilatation_borne_sur_les_Besov} implies that
$\|g\|_{\besovect{p,p}{-s}{X}} \equival
\|(zg)'\|_{\besovect{p,p}{-s-1}{X}}$. Applied to $g(z)=f(z)/z$ (recall that
$f(0)=0$) this inequality becomes $\|f'\|_{\besovect{p,p}{-s-1}{X}} \equival
\|z\mapsto f(z)/z\|_{\besovect{p,p}{-s}{X}}$. The inequality 
\[\|z\mapsto f(z)/z\|_{\besovect{p,p}{-s}{X}} \equival
\|f\|_{\besovect{p,p}{-s}{X}}\] is easy and concludes the proof.
\end{proof}

\section{Operator space valued Hankel matrices}

In this section we finally prove the main results stated in the
Introduction, Theorem \ref{thm=thm_principal_Hankel_Besov}.  In the
particular case when $\alpha=\beta=0$, we recover Theorem
\ref{thm=Main_thm}. We prove the two sides of
\eqref{eq=inegalite_principale_Hankel_Besov} separately. 

For the right-hand side, we first recall a proof for the cases when $p=1$
or $p=\infty$ (this was contained in Peller's proof since for
non-commutative $L^1$ or $L^\infty$ spaces, regularity and complete
boundedness coincide; we will still provide a proof which is more precise
as far as constants are concerned). Then we derive the case of a general
$p$ by an interpolation argument.

The left-hand side inequality is then derived from the right-hand side for
$\alpha=\beta=1$ by duality.

We study the optimality of the bounds in Theorem \ref{thm=Main_thm}, and
finally derive Theorem \ref{thm=projection}.

\subsection{Right hand side of \eqref{eq=inegalite_principale_Hankel_Besov}
  for $p=1$}
We first prove that for a formal series $\varphi = \sum_{k \geq 0} \widehat
\varphi(k) z^k$ with $\widehat \varphi(k) \in E$, it is sufficient that
$\varphi$ belongs to $\besov{p}{1/p+\alpha + \beta}$ to ensure that
$\Gamma_\varphi^{\alpha,\beta} \in S^p[E]$. We first treat the case when
$p=1$.

Let $E$ be an arbitrary operator space. Since (formally) $\varphi =
\sum_0^\infty W_n \conv \varphi$, and
$\|\varphi\|_{\besovect{1}{1+\alpha+\beta}{E}} = \sum_{n \geq 0}2^{n(1+
  \alpha + \beta)} \|W_n \conv \varphi\|_1$, by the triangle
inequality replacing $\varphi$ by $W_n \conv \varphi$ it is enough to
prove that, if $\varphi = \sum_{k=0}^m a_k z^k$ with $a_k \in E$,
\[ \left\|\Gamma_\varphi^{\alpha,\beta} \right\|_{S^1[E]} \leq C
\frac{(1+m)^{1+\alpha+\beta}}{\sqrt{(\alpha+1/2)(\beta+1/2)}}
\|\varphi\|_{L^1(\tore;E)}.\]

But we can write
\[ \Gamma_\varphi^{\alpha,\beta} = \int_\tore \left( \varphi(z)
  (1+j)^\alpha (1+k)^\beta \overline z^{j+k} \right)_{0 \leq j,k\leq m}
dz\] and compute, for $z \in \tore$,
\begin{multline*}
  \left\| \left( \varphi(z) (1+j)^\alpha (1+k)^\beta \overline z^{j+k}
    \right)_{0 \leq j,k\leq m}\right \|_{S^1[E]}\\ = \|\varphi(z)\|_E
  \left\| \left( (1+j)^\alpha (1+k)^\beta \overline z^{j+k} \right)_{0 \leq
      j,k\leq m}\right \|_{S^1},
\end{multline*} 
with
\[\left\|  \left( (1+j)^\alpha (1+k)^\beta \overline z^{j+k} \right)_{0 \leq j,k\leq m}\right
\|_{S^1} = \left\| \left((1+j)^\alpha \right)_{j=0\dots
  m}\right\|_{\ell^2} \left\| \left((1+k)^\beta \right)_{k=0\dots
  m}\right\|_{\ell^2}.\]

Thus the lemma follows from the fact that 
\[ \left\| \left((1+j)^\alpha \right)_{j=0\dots m}\right\|_{\ell^2}^2 \leq
C \frac{(1 + m)^{2 \alpha + 1}}{2 \alpha + 1} \] for a constant $C$ which
depends only on $M = \max\{\alpha,\beta\}$ as long as $\alpha,\beta>-1/2$.

\subsection{Right hand side of \eqref{eq=inegalite_principale_Hankel_Besov}
  for $p=\infty$}

The sufficiency for $p=\infty$ is very similar to easy direction in the
classical proof of Nehari's Theorem that uses the factorization $H^1=H^2
\cdot H^2$, which we first recall. Remember that Nehari's Theorem states
that for any (polynomial function) $\varphi(z)=\sum_{n \geq 0} a_n z^n$
with $a_n \in \C$, $\|\Gamma_\varphi\|_{B(\ell^2)} = \|\varphi\|_{{H^1}^*}$
for the duality $\langle \varphi,f\rangle = \sum_n a_n \widehat f(n)$ for
$f \in H^1(\tore)$. With the notation $f_\xi(z) = \sum_n \xi_n z^n$ for
$\xi = (\xi_n)\in \ell^2$, the inequality $\|\Gamma_\varphi\|_{B(\ell^2)}
\leq \|\varphi\|_{{H^1}^*}$ easily follows from the following elementary
facts:
\begin{enumerate}[a.]
\item \label{enumeration1} For any $\xi =(\xi_n), \eta = (\eta_n) \in \ell^2$, 
  \[\langle \Gamma_\varphi \xi,\eta\rangle_{\ell^2} = \sum_{n \geq 0}
  \widehat \varphi(n) \widehat{ f_\xi f_\eta}(n) = \langle\varphi, f_\xi
  f_\eta \rangle.\]
\item \label{enumeration2} The map $\xi \mapsto f_\xi$ is an isometry
  between $\ell^2$ and $H^2(\tore)$.
\item \label{enumeration3}For any $f_1,f_2 \in H^2(\tore)$, $f_1f_2 \in
  H^1(\tore)$ with norm less than $\|f_1\|_{H^2} \|f_2\|_{H^2}$.
\end{enumerate}

Let us now focus on the right-hand side of inequality
\eqref{eq=inegalite_principale_Hankel_Besov} for $p=\infty$. We fix
$\alpha,\beta>0$ and assume that $E \subset B(H)$ for a Hilbert space $H$.
In this proof we use the fact that $H \hat\otimes \overline H \simeq
B(H)_*$ isometrically through the duality $\scal{T}{\xi \otimes \overline
  \eta} = \scal{T \xi}{\eta}$. For a sequence $x=(\xi_n)$ with $\xi_n$ in
some vector space we also use the notation $f_\xi^{\alpha}(z)$ for the
formal series $\sum_{n \geq 0} (1+n)^\alpha z^n \xi_n$.

Let $\varphi \in \besovect{\infty}{\alpha + \beta}{E}$. We wish to prove
that 
\[\| \Gamma_\varphi^{\alpha,\beta}\|_{B(\ell^2(H))} \leq
C/\min(\alpha,\beta) \|\varphi \|_{\besovect{\infty}{\alpha + \beta}{E}}.\]
Since $\besovect{\infty}{\alpha + \beta}{E}$ is naturally isometrically
contained in $\besovect{\infty}{\alpha + \beta}{B(H)}$ which is (by Theorem
\ref{thm=proprietes_des_besov} and the identification $H \hat\otimes
\overline H \simeq B(H)_*$) isomorphic to the dual space of
$\besovect{1}{-\alpha - \beta}{H \widehat \otimes \overline H}$, we are are
left to prove that
\[ \| \Gamma_\varphi^{\alpha,\beta}\|_{B(\ell^2(H))} \leq
C/\min(\alpha,\beta) \|\varphi \|_{\besovect{1}{-\alpha - \beta}{H \widehat
    \otimes \overline H}^*}.\]
As above this inequality follows immediately from the following three
facts:
\begin{enumerate}[a'.]
\item \label{enumerationp1}For any $\xi =(\xi_n) \in \ell^2(H), \eta =
  (\eta_n) \in \ell^2(H)$,
  \[\langle \Gamma_\varphi^{\alpha,\beta} \xi,\eta\rangle_{\ell^2(H)} = \sum_{n \geq 0}
  \langle \widehat \varphi(n), \widehat{ f_\xi^\beta \otimes f_{\overline
      \eta}^\alpha}(n)\rangle_{B(H),H \widehat \otimes \overline H} =
  \langle \varphi, f_\xi^\beta \otimes f_{\overline \eta}^\alpha\rangle.\]
\item \label{enumerationp2} The map $\xi\in \ell^2(H) \mapsto f_\xi^\beta$
  (resp. $\overline \eta = (\overline \eta_n) \in \ell^2(\overline H) \mapsto
  f_{\bar\eta}^\alpha$) is an isomorphism between $\ell^2(H)$ and
  $\besovect{2}{-\beta}H$ (resp. between $\ell^2(\overline H)$ and
  $\besovect{2}{-\alpha}{\overline H}$). Moreover the constants in these
  isomorphisms depend only on $M=\max(\alpha,\beta)$.
\item \label{enumerationp3}For any $f \in \besovect{2}{-\beta}{H}$ and
  $g \in \besovect{2}{-\alpha}{\overline H}$, the series $f \otimes g
  \in \besovect{1}{-\alpha - \beta}{H \widehat \otimes H}$ and moreover
  there is a constant $C$ depending only on $M$ such that
  \[\|f \otimes g\|_{\besovect{1}{-\alpha - \beta}{H \hat\otimes \overline
      H}} \leq \frac{C}{\min(\sqrt \alpha,\sqrt \beta)}
  \|f\|_{\besovect{2}{-\beta}{H}} \|g\|_{\besovect{2}{-\alpha}{\overline
      H}}.
\]
\end{enumerate}

The facts (\ref{enumerationp1}') and (\ref{enumerationp2}') are again
elementary while fact (\ref{enumerationp3}') is not and follows from the
properties of Besov spaces stated in the previous section. Let us prove it.
\begin{rem} In fact the same holds with $H$ and $\overline H$ replaced by
  arbitrary Banach spaces, but in this case one has to replace $C/\min(
  \sqrt \alpha, \sqrt \beta)$ by $C/\min( \alpha, \beta)$.
\end{rem}
\begin{proof}[Proof of (\ref{enumerationp3}')]
From Theorem \ref{thm=besov_comme_analytiques_avec_derivees},
\[ \|f \otimes g\|_{\besovect{1}{-\alpha - \beta}{H \hat\otimes
    \overline H}} \equival |f(0)||g(0)| + \left\|(1-|z|)^{\alpha + \beta} (f
  \otimes g)' \right\|_{L^1(\disc,dz;H \hat\otimes \overline H)}.\] 
Since $(f \otimes g)' = f'\otimes g + f \otimes g'$, (\ref{enumerationp3}')
will clearly follow from the existence of a constant $C$ depending on $M$
only such that
\begin{equation*}
\left\|(1-|z|)^{\alpha + \beta} f' \otimes g \right\|_{L^1(\disc,dz;H
  \hat\otimes \overline H)} \leq \frac C {\sqrt \alpha}
\|f\|_{\besovect 2 {-\beta}{H}} \|g\|_{\besovect 2 {-\alpha}{\overline H}}.
\end{equation*}
But by the Cauchy-Schwarz inequality, we get that
\begin{multline*}
\left\|(1-|z|)^{\alpha + \beta} f'  \otimes g \right\|_{L^1(\disc,dz;H
  \hat\otimes \overline H)} 
\\ 
\leq \left\|(1-|z|)^{\beta + 1/2 }
  f'\right\|_{L^2(\disc,dz;H)} \left\|(1-|z|)^{\alpha - 1/2 }
  g\right\|_{L^2(\disc,dz;\overline H)}
\end{multline*}
For the first term, use again Theorem
\ref{thm=besov_comme_analytiques_avec_derivees} to get
\begin{equation}
\label{eq=first_inequality_for_fprime}
 \left\|(1-|z|)^{\beta + 1/2 } f'\right\|_{L^2(\disc,dz;H)} \equival
 \|f\|_{\besovect{2}{-\beta}{H}}, \end{equation} whereas for the
second term Theorem \ref{thm=besov_comme_analytiques_casHilbert}
implies
\[\left\|(1-|z|)^{\alpha - 1/2 }
  g\right\|_{L^2(\disc,dz;\overline H)} \equival \frac 1 {\sqrt \alpha}
  \|g\|_{\besovect 2 {-\alpha}{\overline H}}.\]
\end{proof}

\subsection{Right hand side of  \eqref{eq=inegalite_principale_Hankel_Besov} for a general $p$.}
Let us first reformulate the right-hand side of
\eqref{eq=inegalite_principale_Hankel_Besov}.

Denote by $D$ the infinite diagonal matrix $D_{j,j} = 1/(1+j)$ and
$D_{j,k}=0$ if $j \neq k$.  Let $p$, $\alpha$ and $\beta$ as in
Theorem \ref{thm=thm_principal_Hankel_Besov}. Define $\widetilde
\alpha=\alpha+1/2p$ and $\widetilde \beta=\beta+1/2p$. Then for any
$\varphi$ 
\[\Gamma_\varphi^{\alpha,\beta} = D^{\nicefrac 1 {2p}} 
\Gamma_\varphi^{\widetilde \alpha, \widetilde \beta} D^{\nicefrac 1
  {2p}}, \] and Theorem \ref{thm=dilatation_borne_sur_les_Besov} implies
that the map $I_{\widetilde \alpha+ \widetilde \beta} :
\besov{p}{\widetilde \alpha+ \widetilde \beta} \to \besov{p}{0}$ is a
regular isomorphism (with regular norms of the map and its inverse
depending only on $\max(\alpha,\beta)$).

The main result of this section is 
\begin{lemma}
\label{thm=application_reguliere}
Let $M>0$. Take $0<\alpha,\beta<M$ and $1\leq p \leq \infty$. The map
\begin{eqnarray*}
T_p: \besov{p}{0} &\to & S^p \textrm{ (or $B(\ell^2)$ if $p=\infty$)}\\ \varphi & \mapsto & D^{\nicefrac 1
  {2p}} \left(\widehat \varphi(j+k)
\frac{(1+j)^\alpha(1+k)^\beta}{(1+j+k)^{\alpha+\beta}}
\right)_{j,k\geq 0}D^{\nicefrac 1 {2p}}
\end{eqnarray*}
is regular, with regular norm less that $C/ (\min (\alpha,\beta))^{1/2+1/2p}$
for some constant $C$ depending only on $M$.
\end{lemma}
As explained above, this result is equivalent to the right-hand side
inequality in \eqref{eq=inegalite_principale_Hankel_Besov}. More precisely
the above Theorem for some $\alpha,\beta>0$ and $1 \leq p \leq \infty$ is
equivalent to the right-hand side inequality in
\eqref{eq=inegalite_principale_Hankel_Besov} for the same $p$ but with
$\alpha$ and $\beta$ replaced by $\alpha-1/2p$, $\beta-1/2p$. In the proof
below, Pisier's Theorem \ref{thm=applications_regulieres_NC} on
interpolation of regular operators is used, but the reader unfamiliar with
regular operators can as well directly use Stein's complex interpolation
method with vector-valued Besov spaces and Schatten classes.

\begin{proof}[Proof of Lemma \ref{thm=application_reguliere}]
We have already seen that the map $T_p$ is regular(=completely
bounded) when $p=1$ or $p=\infty$. Therefore up to the change of
density given by $D$, $T_p$ is simultaneously completely bounded on
$\besov{1}{0}$ and $\besov{\infty}{0}$, which should imply that $T_p$
is regular.

To check this more rigorously, we use Pisier's Theorem
\ref{thm=applications_regulieres_NC}. Since the Besov space
$\besov{p}{0}$ is a complemented subspace of $L^p(\N \times \tore)$
(where $\N \times \tore$ is equipped with the product of the counting
measure on $\N$ and the Lebesgue measure on $\tore$), and since the
complementation map $P$ is regular and is the same for every $p$,
$T_p$ naturally extends to a map $T_p \circ P:L^p(\N \times \tore) \to
S^p$ which is still completely bounded for $p=1, \infty$.

To show that $T_p$ is regular, we show that $T_p \circ P \in
\left[CB(L^\infty,B(\ell^2)), CB(L^1,S^1)\right]_\theta$ (where the
first $L^\infty$ and $L^1$ spaces are $L^\infty(\N \times \tore)$ and
$L^1(\N \times \tore)$). Since by the equivalence theorem for complex
interpolation $[A_0,A_1]_\theta \subset [A_0,A_1]^\theta$ with
constant $1$ for any compatible Banach spaces $A_0,A_1$ (Theorem 4.3.1
in \cite{MR0482275}), Theorem \ref{thm=applications_regulieres_NC}
will imply that $T_p \circ P$ is regular and hence its restriction to
$\besov{p}{0}$, $T_p$, too.

Consider the analytic map $f(z)$ with values in $CB(L^1,S^1)+
CB(L^\infty,B(\ell^2))$ given by $f(z) = D^{z/2} T_\infty \circ P
D^{z/2}$ ($f$ takes in fact values in $CB(L^\infty,B(\ell^2))$). Then
$f(1/p) = T_p \circ P$. The conjugation by a unitary is a complete
isometry on $B(\ell^2)$ and on $S^1$. Therefore if $Re(z)=0$,
$\|f(z)\|_{CB(L^\infty,B(\ell^2))} = \|T_\infty \circ
P\|_{CB(L^\infty,B(\ell^2))} \leq C/\sqrt{\min(\alpha, \beta)}$ and if $Re(z)=1$,
$\|f(z)\|_{CB(L^1,S^1)} = \|T_1\circ P\|_{CB(L^1,S^1)} \leq
C/\sqrt{\alpha \beta} \leq C/\min(\alpha,\beta)$.
This proves that 
\[\left\| T_p \right\|_{B_r(L^p,S^p)} \leq C/ (\min
(\alpha,\beta))^{1/2+1/2p}.\]%
\end{proof}

\subsection{Left-hand side of \eqref{eq=inegalite_principale_Hankel_Besov}}
\label{part=ineg_gauche}
In this section we assume that the right-hand side of
\eqref{eq=inegalite_principale_Hankel_Besov} holds for
$\alpha=\beta=1$, that is to say the operator
\begin{eqnarray*}\besov{p}{1/p + 2} & \to & S^p\\
\varphi &\mapsto & \Gamma_\varphi^{1,1}
\end{eqnarray*}
is regular for every $1\leq p \leq \infty$.

Fix now $1 \leq p \leq \infty$ and $\alpha,\beta > -1/2p$. We prove that
the map $\Gamma_\varphi^{\alpha,\beta} \mapsto \varphi$ is regular from the
subspace of $S^p$ (or $B(\ell^2)$ if $p=\infty$) formed of all the matrices
of the form $\Gamma_\varphi^{\alpha,\beta}$ to
$\besov{p}{1/p+\alpha+\beta}$.

For $\psi \in \besov{p'}{1/p'+2}$ define the matrix
\begin{eqnarray*}
  \widetilde \Gamma_\psi^{1,1} & = & \left( \frac{D_j^{\alpha +
      1}}{(1+j)^\alpha} \frac{D_k^{\beta + 1}}{(1+k)^\beta} \widehat
  \psi (j+k)\right)_{j,k\geq 0}\\ & = & diag\left(\frac{D_j^{\alpha +
      1}}{(1+j)^{\alpha+1}}\right) \cdot \Gamma_\psi^{1,1} \cdot
  diag\left(\frac{D_k^{\beta + 1}}{(1+k)^{\beta+1}}\right).
\end{eqnarray*}
First note that since $\sup_{-1/2 \leq \alpha \leq M} \sup_{j\geq 0}
D_j^{\alpha+1}/(1+j)^{\alpha+1} < \infty$ the assumption with $p'$
implies that the operator $T:\psi \mapsto \widetilde
\Gamma_\psi^{1,1}$ is also regular from $\besov{p'}{1/p' + 2}$ to
$S^{p'}$ with regular norm bounded by some constant depending only on
$M$.

Recall that by Theorem \ref{thm=proprietes_des_besov}
$\besov{p}{-1/p'-2} \simeq (\besov{p'}{1/p' + 2})^*$ if $p >1$ (and
$(\besov{p}{-1/p'-2})^* \simeq \besov{p'}{1/p' + 2}$ if
$p<\infty$). Since $\besov{p}{1/p-3}$ is complemented in
$\ell_{p}^{1/p -3}(\N;L^{p})$ with a regular complementation map,
Theorem \ref{thm=regularity_au_dual} implies that the dual map $T^* :
S^{p} \to \besov{p}{-1/p'-2} = \besov{p}{1/p-3}$ is also regular.

It is now enough to compute explicitly the restriction of $T^*$ to the set
of matrices of the form $\Gamma_\varphi^{\alpha,\beta}$ to conclude. Indeed
for any analytic $\varphi: \tore \to \C$ such that
$\Gamma_\varphi^{\alpha,\beta} \in S^p$ (or $B(\ell^2)$), and any $\psi \in
\besov{p'}{1/p' + 2}$ we have
\begin{eqnarray*}
\left \langle T^* \Gamma_\varphi^{\alpha,\beta},\psi\right \rangle &=&
\left \langle \Gamma_\varphi^{\alpha,\beta},T \psi\right \rangle\\ 
&=& \sum_{j,k\geq 0} D_j^{\alpha + 1} D_k^{\beta + 1} \widehat
\varphi(j+k) \widehat \psi(j+k)\\ 
& = & \sum_{n \geq 0} D_n^{\alpha+\beta+3} \widehat
  \varphi(n)\widehat \psi(n)\\ 
& = & \langle\widetilde I_{\alpha+\beta+3}\varphi,\psi\rangle.
\end{eqnarray*}

We used that for all $\alpha,\beta \in \R$, and all $n \in \N$
\[ \sum_{j + k = n} D_j^\alpha D_k^\beta = D_n^{\alpha +  \beta + 1},\]
which follows from the equality $\sum_{n \geq 0} D_n^\alpha x^n =
(1+x)^{-\alpha -1}$ for $|x|<1$.

Thus we have that $T^* \Gamma_\varphi^{\alpha,\beta} = \widetilde
I_{\alpha+\beta+3}\varphi$. By Theorem
\ref{thm=dilatation_borne_sur_les_Besov} the map $\left(\widetilde
I_{\alpha+\beta+3}\right)^{-1}$ is regular as a map from
$\besov{p}{1/p-3}$ to $\besov{p}{1/p+\alpha+\beta}$. Hence the map
$\Gamma_\varphi^{\alpha,\beta} \mapsto \varphi$ is regular from the
subspace of $S^p$ formed of all the matrices of the form
$\Gamma_\varphi^{\alpha,\beta}$ to $\besov{p}{1/p+\alpha+\beta}$. This
concludes the proof (it is immediate from the proof that the regular
norm of this map only depends on $M$).

\subsection{Optimality of the constants}
\label{part=optimality_of_constants}
In this last part we show that the inequality 
\begin{equation}
\label{eq=inegalite_prinicipale}
C^{-1} \left\|\varphi\right\|_{\besovect{p}{1/p}{E}}\leq
\left\|\Gamma_\varphi \right\|_{S^p[E]} \leq C \sqrt p
\left\|\varphi\right\|_{\besovect{p}{1/p}{E}}
\end{equation} in Theorem
\ref{thm=Main_thm} is optimal even when $E=\C$ (up to constants not
depending on $p$). This observation is due to {\'E}ric Ricard who kindly
allowed to reproduce his proof here.

The fact that the left-hand side of \eqref{eq=inegalite_prinicipale} is
optimal is obvious: indeed if $\varphi(z)=1$ then $\Gamma_\varphi$ is a
rank one orthogonal projection and hence $\left\|\Gamma_\varphi
\right\|_{S^p}=1 = \left\|\varphi\right\|_{\besov{p}{1/p}}$ for any
$p$. 

For the right-hand side inequality consider the positive integer $n$ such
that $n\leq p < n+1$. Let $a_1,\dots,a_n \in \C$ and consider the function
$\varphi_a=\sum_{k=0}^n a_k z^{2^k}$. We clearly have
\[ \left\|\varphi_a\right\|_{\besov{p}{1/p}} = (\sum_{k=0}^n 2^{k}
|a_k|^p)^{1/p} \leq 2^{\nicefrac{n+1} p} \max_{k} |a_k| \leq 4 \max_k
|a_k|,\] and the following lemma therefore implies that the ratio
$\left\|\varphi_a\right\|_{\besov{p}{1/p}} /\left\|\Gamma_{\varphi_a}
\right\|_{S^p}$ can be as small as $12/\sqrt n$, which shows the optimality
of the right-hand side of \eqref{eq=inegalite_prinicipale}.
\begin{lemma}
  For any $1 \leq p \leq \infty$ and any (finite) sequence $a =
  (a_k)_{k\geq 0}$ we have
\[ \left\|\Gamma_{\varphi_a} \right\|_{S^p} \geq  \frac 1 3 \|a\|_{\ell^2}.\]
\end{lemma}
\begin{proof}
  Since $\| \cdot\|_{S^p} \geq \| \cdot\|_{B(\ell^2)}$ for any $1\leq p
  \leq \infty$, and since by Nehari's Theorem
  \[ \|\Gamma_{\varphi_a}\|_{B(\ell^2)} = \| \varphi_a\|_{{H^1}^*},\] the
  statement follows from the inequality $\| \varphi_a\|_{{H^1}^*} \geq
  \|a\|_{\ell^2}/3$, which is the dual inequality of the classical Paley
  inequality
  \[ \left(\sum_{k \geq 0} |\widehat f(2^k)|^2 \right)^{1/2} \leq 3
  \|f\|_{H^1}\] which holds for any $f \in H^1(\tore)$.
\end{proof}

\subsection{The projection}
As in the introduction, $P_{Hank}$ will denote the natural projection from
the space of infinite $\N \times \N$ matrices onto the space of Hankel
matrices. The boundedness properties of $P_{Hank}$ stated in Theorem
\ref{thm=projection} are formal consequences of Theorem \ref{thm=Main_thm}.

\begin{proof}[Proof of Theorem \ref{thm=projection}]
  Let $1<p,p'<\infty$, with $1/p+1/p'=1$. Since for the identification
  $(S^p)^* = S^{p'}$, ${P_{Hank}}^* = P_{Hank}$, we can restrict ourselves
  to the case when $1 < p \leq 2$. We thus have to show that 
  \begin{equation} \label{eq=norme_projection}
\|P_{Hank}\|_{S^p \to S^p} \equival \|P_{Hank}\|_{B_r(S^p,S^p)}
  \equival \sqrt{p'}
\end{equation}
up to constants not depending on $p$.

This follows from Theorem \ref{thm=Main_thm}. More precisely let $T:\psi
\mapsto \Gamma_\psi$ defined from $\besov{p'}{1/p'}$
to $S^{p'}$. Then by Theorem \ref{thm=Main_thm}, we have that  
\[ \|T\|_{\besov{p'}{1/p'} \to S^{p'}} \equival
\|T\|_{B_r(\besov{p'}{1/p'},S^{p'})} \equival \sqrt{p'}.\]

As in part \ref{part=ineg_gauche} this implies (for the natural dualities)
that
\[ \|T^*\|_{S^p \to \besov{p}{-1/p'}} \equival
\|T^*\|_{B_r(S^p,\besov{p}{-1/p'})} \equival \sqrt{p'}.\]

But $T^* (a_{j,k})_{j,k\geq 0} = \sum_{j,k \geq 0} a_{j,k} z^{j+k}$. Thus
we have the following factorization of $P_{Hank}$:
\[ \xymatrix{ {S^p} \ar[r]^{P_{Hank}} \ar[d]^{T^*} & {S^p} \\
  \besov{p}{-1/p'} \ar[r]^{I_{-1}} & {\besov{p}{1/p}} \ar[u]^{T}}.\] This
concludes the proof since $I_1$ (resp. $T$) is a regular isomorphism
bewteen $\besov{p}{-1/p'}$ and $\besov{p}{1/p}$ (resp. between
$\besov{p}{1/p}$ and the subspace of Hankel matrices in $S^p$), and the
regular norms of these isomorphisms as well as their inverses can be
dominated uniformly in $p$ (recall that $1<p \leq 2$).
\end{proof}

\subsection*{Acknowledgement} The author would like to thank his adviser
Gilles Pisier for suggesting the problem. He also thanks Quanhua Xu for
pointing out the result from Lemma
\ref{thm=besov_comme_analytiques_casHilbert}, and {\'E}ric Ricard for
allowing to reproduce his argument in part
\ref{part=optimality_of_constants}, and for other useful discussions
regarding the optimality of the constants.

\bibliographystyle{plain} \bibliography{biblio}

\def\cprime{$'$} \def\cprime{$'$}
\begin{thebibliography}{1}

\bibitem{MR0482275}
J{\"o}ran Bergh and J{\"o}rgen L{\"o}fstr{\"o}m.
\newblock {\em Interpolation spaces. {A}n introduction}.
\newblock Springer-Verlag, Berlin, 1976.
\newblock Grundlehren der Mathematischen Wissenschaften, No. 223.

\bibitem{MR602274}
V.~V. Peller.
\newblock Hankel operators of class {${ S}\sb{p}$} and their applications
  (rational approximation, {G}aussian processes, the problem of majorization of
  operators).
\newblock {\em Mat. Sb. (N.S.)}, 113(155)(4(12)):538--581, 637, 1980.

\bibitem{MR725454}
V.~V. Peller.
\newblock Description of {H}ankel operators of the class {${ S}\sb{p}$} for
  {$p>0$}, investigation of the rate of rational approximation and other
  applications.
\newblock {\em Mat. Sb. (N.S.)}, 122(164)(4):481--510, 1983.

\bibitem{MR647702}
Vladimir~V. Peller.
\newblock Vectorial {H}ankel operators, commutators and related operators of
  the {S}chatten-von {N}eumann class {$\gamma \sb{p}$}.
\newblock {\em Integral Equations Operator Theory}, 5(2):244--272, 1982.

\bibitem{MR1949210}
Vladimir~V. Peller.
\newblock {\em Hankel operators and their applications}.
\newblock Springer Monographs in Mathematics. Springer-Verlag, New York, 2003.

\bibitem{MR1324838}
Gilles Pisier.
\newblock Regular operators between non-commutative {$L\sb p$}-spaces.
\newblock {\em Bull. Sci. Math.}, 119(2):95--118, 1995.

\bibitem{MR1648908}
Gilles Pisier.
\newblock Non-commutative vector valued {$L\sb p$}-spaces and completely
  {$p$}-summing maps.
\newblock {\em Ast\'erisque}, (247):vi+131, 1998.

\end{thebibliography}

\end{document}